\documentclass[12pt,twoside,reqno]{amsart}
\usepackage{amsmath, amsfonts, amsthm, amssymb}
\usepackage{graphicx}
\usepackage{float}
\usepackage{verbatim}
\usepackage[utf8]{inputenc}
\usepackage[T1]{fontenc}
\usepackage[colorlinks = true, citecolor = black, linkcolor = black, urlcolor = black, pdfstartview=FitH]{hyperref}  

\usepackage[left=2.3cm,top=3cm,right=2.3cm]{geometry}

\numberwithin{equation}{section}

\theoremstyle{plain}
\newtheorem{thm}{Theorem}[section]
\newtheorem{theorem}[thm]{Theorem}

\newtheorem{lemma}[thm]{Lemma}

\newtheorem{problem}[thm]{Problem}
\newtheorem*{question*}{Question}

\theoremstyle{definition}

\newtheorem{definition}[thm]{Definition}

\renewcommand{\geq}{\geqslant}
\renewcommand{\leq}{\leqslant}

\renewcommand{\P}{\text{P}}

\newcommand{\R}{\mathbb{R}}
\newcommand{\C}{\mathbb{C}}

\newcommand{\N}{\mathbb{N}}
\renewcommand{\P}{\mathbb{P}}
\newcommand{\E}{\mathbb{E}}

\newcommand{\cF}{\mathcal{F}}

\newcommand{\eps}{\varepsilon}

\newcommand{\Hd}{\dim_\mathrm{H}}

\newcommand{\Fd}{\dim_\mathrm{F}}

\newcommand{\Z}{\mathbb{Z}}

\renewcommand{\epsilon}{\varepsilon}
\renewcommand{\rho}{\varrho}
\renewcommand{\phi}{\varphi}

\renewcommand{\t}{\mathrm{\mathbf{t}}}
\newcommand{\s}{\mathrm{\mathbf{s}}}
\renewcommand{\k}{\mathrm{\mathbf{k}}}
\newcommand{\n}{\mathrm{\mathbf{n}}}

\usepackage[usenames,dvipsnames]{xcolor}

\title[On the Fourier analytic structure of the Brownian graph]{On the Fourier analytic structure \\ of the Brownian graph}

\author{Jonathan M. Fraser}
\address{Jonathan M. Fraser\\School of Mathematics \& Statistics\\University of St Andrews\\ St Andrews\\ KY16 9SS\\ UK  }
\email{jmf32@st-andrews.ac.uk}

\author{Tuomas Sahlsten}
\address{Tuomas Sahlsten\\School of Mathematics\\ University of Manchester\\Oxford Road \\Manchester\\M13 9PL, UK}
\email{tuomas.sahlsten@manchester.ac.uk}


\thanks{JMF acknowledges financial support from the Leverhulme Trust (RF-2016-500) and TS acknowledges the support from the European Union (ERC grant $\sharp$306494 and Marie Sk{\l}odowska-Curie Individual Fellowship grant $\sharp$655310)}

\begin{document}

\begin{abstract} In a previous article (\textit{Int. Math. Res. Not.} 2014, 2730--2745) T. Orponen and the authors proved that the Fourier dimension of the graph of any real-valued function on $\R$ is bounded above by $1$. This partially answered a question of Kahane ('93) by showing that the graph of the Wiener process $W_t$ (Brownian motion) is almost surely not a Salem set.  In this article we complement this result by showing that the Fourier dimension of the graph of $W_t$ is almost surely $1$. In the proof we introduce a method based on It\=o calculus to estimate Fourier transforms by reformulating the question in the language of It\=o drift-diffusion processes and combine it with the classical work of Kahane on Brownian images.
\end{abstract}

\maketitle

\section{Introduction and results}

\subsection{Geometric properties of Brownian motion}

Gaussian processes are standard models in modern probability theory and perhaps the most well-studied example is the \textit{Wiener process} (or standard \textit{Brownian motion}) $W = W_t : \mathbb{R}^{\geq 0} \to \mathbb{R}$ characterised by the properties: $W_0 = 0$, the map $t \mapsto W_t$ is almost surely continuous, and $W_t$ has independent increments such that $W_t - W_s$ for $t > s$ is normally distributed: 
$$W_t - W_s \sim N(0,t-s).$$ 
The Wiener process has far reaching importance throughout mathematics and it is a topic of particular interest to understand its geometric structure. This can be achieved by studying several random fractals associated to the process such as \textit{images} $W(K) := \{W_t : t \in K\}$ of compact sets $K \subset [0,\infty)$, \textit{level sets} $L_c(W) := \{t \in \R : W_t = c\}$ for $c \in \R$, \textit{graphs} $G(W) := \{(t,W_t) : t \in \R\}$ and other more delicate constructions such as $\mathrm{SLE}_\kappa$-curves. 

The basic properties of Brownian motion mean that these random fractals enjoy a certain `statistical self-similarity' which facilitates computation of their Hausdorff dimensions $\Hd$. Classical results include McKean's proof \cite{mckean} that $\Hd W(K) = \min\{1,2\Hd K\}$ almost surely for each compact $K \subset [0,\infty)$. Moreover, for the level sets $\Hd L_c(W) = 1/2$ almost surely for $c = 0$ by Taylor \cite{taylor2} and for all $c \in \R$ by Perkins \cite{perkins} conditioned on $L_c(W)$ being non-empty. For the Brownian graph $G(W)$, Taylor \cite{taylor} proved that $\Hd G(W) = 3/2$ almost surely and Beffara computed the Hausdorff dimensions of $\mathrm{SLE}_\kappa$-curves \cite{beffara}.  Moreover, Hausdorff dimensions for similar sets given by many other Gaussian processes, such as fractional Brownian motion, have been also considered, see for example Adler's classical results \cite{adler} for fractional Brownian graphs and the recent work concerning variable drift by Peres and Sousi \cite{peressousi}.

\subsection{Fourier analytic properties of Brownian motion} The Hausdorff dimension is the most commonly used tool for measuring the size of a set $A$ but there is also another fundamental notion based on Fourier analysis which reveals more arithmetic and geometric features of $A$, including curvature, which are not seen by the Hausdorff dimension. This is based on studying the \textit{Fourier coefficients} of a probability measure $\mu$ on  $A \subset \R^d$, which are defined by
$$\widehat{\mu}(\xi) := \int e^{-2\pi i\xi \cdot x}\, d\mu(x), \quad \xi \in \R^d.$$
Now the size of $A$ can be linked to the existence of probability measures $\mu$ on $A$ with decay of Fourier coefficients $\widehat{\mu}(\xi)$ when $|\xi| \to \infty$. The following connection between Hausdorff dimension and decay of Fourier coefficients is well-known and goes back to Salem and Kaufman, but we refer the reader to \cite{mattilaFourier} for the details.  If $\Hd A > s$, then $A$ supports a probability measure $\mu$ with $|\widehat{\mu}(\xi)| = O(|\xi|^{-s/2})$ ``on average'', that is, $\int_{\R^d} |\widehat{\mu}(\xi)|^2 |\xi|^{s-d} \, d\xi < \infty$ and vice versa the Hausdorff dimension can be bounded from below if such a measure $\mu$ can be found. It is possible, however, that $\Hd A = s > 0$ but no measure $\mu$ on $A$ has Fourier decay at infinity, this happens for example when $A$ is the middle-third Cantor set in $\R$. Therefore, one defines the notion of \textit{Fourier dimension} $\Fd A$ of a set $A \subset \mathbb{R}^d$ as the supremum of $s \in [0,d]$ for which there exists a probability measure $\mu$ supported on $A$ such that
\begin{align}\label{eq:poly}
|\widehat\mu(\xi)| = O(|\xi|^{-s/2}), \quad \text{as } |\xi| \to \infty.
\end{align}
Then by this definition we always have $\Fd A \leq \Hd A$ and if the two dimensions coincide then $A$ is called a \textit{Salem set} or a \textit{round set} after Kahane \cite{kahane}. In general Fourier dimension and Hausdorff dimension have no relationship other than this; in fact, K\"orner \cite{korner} established that for any $0 \leq s < t \leq 1$ it is possible to construct examples $A \subset \R$ with $\Fd A = s$ and $\Hd A = t$. Further properties of Fourier dimension were recently developed by Ekstr\"om, Persson and Schmeling \cite{EPS}. For a more in depth account of Fourier dimension, the reader is referred to \cite{mattila, mattilaFourier}.

Finding measures $\mu$ on $A$ with polynomially decaying Fourier transform (i.e. \eqref{eq:poly} for some $s > 0$) has deep links to absolute continuity, arithmetic and geometric structure, and curvature. If $A$ supports a measure $\mu$ such that \eqref{eq:poly} holds with $s > 1$, then Parseval's identity yields that $\mu$ is absolutely continuous to Lebesgue measure and $A$ must contain an interval. An application of Weyl's criterion known as the Davenport-Erd\"os-LeVeque criterion (see \cite{DEL}) yields that in $\R$  polynomial decay of $\widehat{\mu}$ guarantees that $\mu$ almost every number is normal in every base and a interesting result of {\L}aba and Pramanik \cite{laba} shows that if the $s$ in \eqref{eq:poly} is sufficiently close to $1$ for a Frostman measure $\mu$ on $A \subset \R$ and there is a suitable control over the constants (see the recent work of Shmerkin \cite{Shm}), then $A$ contains non-trivial $3$-term arithmetic progressions. Moreover, an analogous result also holds for higher dimensions with arithmetic patches \cite{chan}. 

On the curvature side, if $A$ is a line-segment in $\R^2$, then $A$ cannot contain any measure with Fourier decay at infinity so $A$ cannot be a Salem set. However, if $A$ is an arc of a circle or more generally a $1$-dimensional smooth manifold with non-vanishing curvature then the $1$-dimensional Hausdorff measure $\mu$ on $A$ satisfies \eqref{eq:poly} with $s = 1$, see \cite{mattilaFourier}. In particular, $A$ is a Salem set. In these examples of $A$ one can observe that the important arithmetic or curvature features present are not seen from the Hausdorff dimension.

Constructing explicit Salem sets (which are not manifolds), or just sets $A$ supporting a measure $\mu$ satisfying \eqref{eq:poly} for some $s > 0$, can be achieved through, for example, Diophantine approximation by Kaufman's works \cite{kaufman1, kaufman2}, Bluhm \cite{bluhm}, Queff\'{e}lec-Ramar\'{e} \cite{QR} or via thermodynamical tools by Jordan and Sahlsten \cite{JordanSahlsten2015}. However, for random sets it has been observed in many instances that $A$ is either almost surely Salem or at least supports a measure $\mu$ with \eqref{eq:poly} for some $s > 0$. This was first done for  random Cantor sets by Salem \cite{salem}, where Salem sets were also introduced. Later Kahane published his classical papers \cite{kahane1966a,kahane1966b}, where he found out that the Wiener process and other Gaussian processes provide natural examples. 

Since Kahane and Salem, the study of Fourier analytic properties of natural sets derived from Gaussian processes and more general random fields has been an active topic. For the Brownian images Kahane \cite{Ka1} proved that for any compact $K \subset \R$ the image $W(K)$ is almost surely a Salem set of Hausdorff dimension $\min\{1,2\dim K\}$. Kahane also established a similar result for fractional Brownian motion. {\L}aba and Pramanik \cite{laba} then applied these to the additive structure of Brownian images. Later Shieh and Xiao \cite{xiao} extended Kahane's work to very general classes of Gaussian random fields. However, understanding the Fourier analytic properties of the level sets and graphs remained an important problem for some time.  In 1993, Kahane \cite{kahane} outlined the problem explicitly.
\begin{problem}[Kahane]
Are the graph and level sets of a stochastic process such as fractional Brownian motion Salem sets?
\end{problem}

This precise formulation of the problem was given by Shieh and Xiao \cite[Question 2.15]{xiao}, but they attribute the problem to Kahane. For the Wiener process Kahane \cite{kahane1983} had already established that the level sets $L_c(W)$ are Salem almost surely for any fixed $c \in \R$ conditioned on $L_c(W)$ being non-empty. The fractional Brownian motion case has recently been considered for $c = 0$ by Fouch\'e and Mukeru \cite{FM}. 

\begin{figure}[ht!]\label{fig:browngraph}
\includegraphics[scale=1.2]{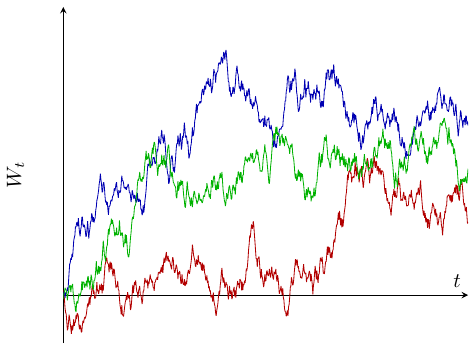}
\caption{Three realisations of the graph $G(W)$ for the Brownian motion $W_t$.}
\end{figure}

Kahane's problem for graphs, even in the case of the standard Brownian motion $W_t$, however, remained open for quite a while until, together with T. Orponen, we established that the Brownian graph $G(W)$ is almost surely \emph{not} a Salem set \cite{prequel}. It turned out that the reason for this is purely geometric: the proof was based on the following application of a Fourier-analytic version of Marstrand's slicing lemma.

\begin{theorem}[Theorem 1.2 in \cite{prequel}]\label{thm:prequel}
For any function $f : [0,1] \to \R$ the Fourier dimension of the graph $G(f)$ cannot exceed $1$.
\end{theorem}

Indeed, since the Hausdorff dimension $\Hd G(W) = 3/2 > 1$ almost surely (see \cite{taylor}), this answers Kahane's problem in the negative for the Wiener process. Note that this also gives a negative answer for fractional Brownian motion since the Hausdorff dimension in that case is also strictly larger than 1 almost surely.

The methods in \cite{prequel} are purely geometric and involve no stochastic properties of Brownian motion.  They also do not shed any light on the precise value for the Fourier dimension of $G(W)$. Note that even though  $\Hd G(f) \geq 1$ for any continuous $f : [0,1]\to\R$, the Fourier dimension of a graph may take any value in the interval $[0,1]$, see \cite{prequel}. For example, $\Fd G(f) = 0$ if $f$ is affine and, moreover, $\Fd G(f) = 0$ for the \textit{Baire generic} $f \in C[0,1]$, see  \cite[Theorem 1.3]{prequel}.

The main result of this paper is to complete the work initiated by Kahane's problem in the case of Brownian motion by establishing the precise almost sure value of the Fourier dimension of $G(W)$. 

\begin{theorem}
\label{thm:main}
The graph $G(W)$ has Fourier dimension $1$ almost surely.
\end{theorem}

Moreover, the random measure $\mu$ we use to realise the Fourier dimension is Lebesgue measure $dt$ on $[0,1]$ lifted onto the graph $G(W)$ via the mapping $t \mapsto (t,W_t)$.  The precise estimate we obtain is that almost surely 
\begin{align}\label{eq:fourierdecay}|\widehat{\mu}(\xi)| = O( |\xi|^{-\frac{1}{2}} \sqrt{\log |\xi|}), \quad \text{as } |\xi| \to \infty,\end{align}
which combined with Theorem \ref{thm:prequel} yields Theorem \ref{thm:main}.  

A natural direction in which to continue this line of research would be to study other Gaussian processes with different covariance structure such as the fractional Brownian motion.

\subsection{Methods: It\=o calculus and reduction to Brownian images}

The key method we introduce to estimate the Fourier transform of the graph measure $\mu$  is based on \textit{It\=o calculus}, which has previously been a natural framework in the theory of stochastic differential equations. As far as we know, It\=o calculus has not been previously considered in this Fourier analytic context. Here we discuss this method and give a brief summary of the main steps in the proof. When written in polar coordinates, \eqref{eq:fourierdecay} asks about the rate of decay for the integral
$$\widehat{\mu}(\xi) = \int_0^1 \exp(-2\pi i u(t\cos \theta + W_t\sin \theta)) \, dt$$
for $\xi = u(\cos \theta,\sin\theta)\in \R^2$, $u > 0$, $\theta \in [0,2\pi)$, as $u \to \infty$. There are two distinct cases  we will consider depending on the direction of $\xi$, which we give a heuristic description of here.

If we ignore the random component $W_t \sin \theta$, that is, set $\theta = 0$ or $\pi$, then standard integration using the chain rule shows that $\widehat{\mu}(\xi)$ equals the Fourier transform of Lebesgue measure $dt$ at $u$, which decays to $0$ with the polynomial rate $u^{-1} = |\xi|^{-1}$, so we are done for these directions. However, if $\theta$ is not equal to $0$ or $\pi$, we still have a small random (non-smooth) term $W_t \sin\theta$, so a classical change of variable formula or other tools from classical analysis cannot be used.
 
 The key observation is that we can write $\widehat{\mu}(\xi) = \int \exp(iX_t) \, dt$, where the stochastic process $X_t$ satisfies the stochastic differential equation
 $$dX_t := b dt + \sigma dW_t$$
identifying it as a so called \textit{It\=o drift-diffusion process}, where $b := -2\pi u \cos \theta$ is the \textit{drift coefficient} of $X_t$ and $\sigma := -2\pi u \sin \theta$ is the \textit{diffusion coefficient} of $X_t$. Such processes have many useful analytic tools from It\=o calculus (see Section \ref{sec:ito}) associated to them, in particular \textit{It\=o's lemma}, which works as an analogue for the chain rule. The price we pay is that It\=o's lemma introduces some multiplicative error terms involving stochastic integrals, but they can be estimated with other tools from It\=o calculus using moment analysis.

The estimates we obtain from It\=o calculus allow us to obtain the correct Fourier decay \eqref{eq:fourierdecay} for $\mu$ when $\theta$ is close to $0$ or $\pi$ with respect to $u^{-1}$ (more precisely, $|\sin \theta| < u^{-1/2}$). In other words, when $\xi$ is close to pointing in the horizontal directions. Thus another estimate is needed for $\theta$ bounded away from $0$ and $\pi$. This is where Kahane's classical work \cite{Ka1} on Brownian images comes into play. If we completely ignore the deterministic component $t \cos \theta$, by setting $\theta = \pi/2$ or $3\pi/2$, then $\widehat\mu(\xi)$ is the Fourier transform the Brownian image measure $\nu$, that is the $t \mapsto W_t$ push-forward of the Lebesgue measure $dt$ on $[0,1]$ at $u$. Kahane \cite{Ka1} in fact already established that the decay of $|\widehat{\nu}(u)|$ is almost surely of the order $u^{-1} \sqrt{\log u} = |\xi|^{-1} \sqrt{\log|\xi|}$ so \eqref{eq:fourierdecay} holds for these directions. A modification  of Kahane's argument reveals that whenever $\theta \neq 0$ or $\pi$, then almost surely
\begin{align*}|\widehat{\mu}(\xi)| = O(|\sin \theta|^{-1}|\xi|^{-1} \sqrt{\log|\xi|})\end{align*}
see the discussion in Section \ref{sec:vertical}. Now one notices that when $\theta$ approaches $0$ or $\pi$, this estimate blows up, and so one cannot obtain a uniform estimate over all directions from this. However, this gives \eqref{eq:fourierdecay} if $|\sin \theta| \geq u^{-1/2}$, so combining with the estimates we obtained through It\=o calculus, we are done. See Section \ref{sec:proof} for more details on the main steps of the proof.  

\subsection{Other measures on the Brownian graph}

Theorem \ref{thm:main} and \eqref{eq:fourierdecay} gives Fourier decay for the push-forward of the Lebesgue measure on $[0,1]$ onto the graph $G(W)$. It would be an interesting problem to see if one can have similar results for other, possibly fractal, measures on $[0,1]$. A possible problem could be:

\begin{problem} Classify measures $\tau$ on $[0,1]$ such that for some $0 < s \leq 1$ we have
$$|\widehat{\tau}(\xi)| = O(|\xi|^{-s/2}), \quad |\xi| \to \infty,$$
and their lift $\mu_\tau$ onto the graph of $G(W)$ under $t \mapsto (t,W_t)$ satisfies
$$|\widehat{\mu_\tau}(\xi)| = O(|\xi|^{-s'/2}), \quad |\xi| \to \infty$$
for any $s'<s$.
\end{problem}

This is motivated by the fact that in Kahane's work \cite{Ka1} it is possible to transfer information on the Fourier decay (or Frostman properties) of $\tau$ onto the image measure. Thus for directions $\theta$ bounded away from $0$ and $\pi$ we could still bound $\widehat{\mu_\tau}(\xi)$ using Kahane's work. The main problem in generalising our approach to fractal measures $\tau$ on $[0,1]$ comes from the lack of an appropriate analogue of It\={o} calculus.

\subsection{Organisation of the paper}
 In Section \ref{sec:ito} we give the necessary background from It\=o calculus. In Section \ref{sec:proof} we will give the proof of our main result Theorem \ref{thm:main}. The key estimates are obtained in Section \ref{sec:horizontal} and Section \ref{sec:vertical} corresponding to the two cases discussed above. 

\section{It\={o} calculus}\label{sec:ito}

\subsection{Stochastic integration} \label{sec:stochastic}

In the proof of the main result Theorem \ref{thm:main} we end up studying integrals of the form $\int f(X_t) \, dt$ for some stochastic processes $X_t$ and smooth scalar functions $f$. As standard analysis methods cannot be applied to these integrals, we need theory from stochastic analysis. Stochastic analysis provides a pleasant framework to deal with non-smooth processes, such as the Wiener process $W_t$, and still preserves many of the classical features present in the smooth setting. In this section we discuss the specific tools from \textit{It\=o calculus} which we will rely on. The main references for this section are given in the book \cite{KS}.

Let $(\Omega,\cF,(\cF_t)_{t \geq 0},\P)$ be a filtered probability space, that is, $\cF_t \subset \cF$ is an increasing filtration in $t$. Let $W = W_t$ be the Wiener process adapted to this filtered probability space, that is, $W_t$ is $\cF_t$ measurable and for each $t,s \geq 0$ the increment $W_{t+s}
 - W_t$ is independent of $\cF_t$. We say that an $\R$ or $\C$ valued stochastic process $Z_t$ is \textit{adapted} if it is $\cF_t$ measurable for all $t \geq 0$.  We will say that a real or complex valued adapted process $Z_t$ is $W_t$ \textit{integrable} if the \textit{quadratic variation} $\int_0^T |Z_t|^2 \, dt < \infty$ for any time $T \geq 0$. Given a real valued adapted $W_t$ integrable stochastic process $X_t$, then $\P$ almost surely for any time $T \geq 0$ it is possible to construct a \textit{stochastic integral}
$$\int_0^T X_t \, dW_t$$
of $X_t$ with respect to $W_t$ in the sense of It\=o, see \cite[Chapter 3.2]{KS}. We use the differential notation $dU_t = X_t \, dW_t$ to mean that $\P$ almost surely $U_T - U_0$ is the stochastic integral of $X_t$ with respect to $W_t$ at time $T \geq 0$.

We mainly deal with complex valued stochastic processes so for the sake of convenience we will also define the \textit{complex valued} stochastic integral for a $\C$ valued $W_t$ integrable adapted process $Z_t$ is defined coordinate-wise using real integrals:
$$\int_0^T Z_t \, dW_t := \int_0^T \mathrm{Re}\, Z_t \, dW_t+i\int_0^T \mathrm{Im}\, Z_t \, dW_t,$$
where the real integrals are standard $\R$ valued stochastic integrals with respect to the Wiener process $W_t$. We write $dZ_t := dX_t + idY_t$ for a complex valued process $Z_t = X_t + iY_t$ with $\R$ valued $X_t$ and $Y_t$.

\subsection{It\={o} drift-diffusion processes} \label{sec:itodiffusion}

The main class of adapted processes to which we apply It\=o calculus is given by Wiener processes with drift and diffusion coefficients. These are called It\=o drift-diffusions:

\begin{definition}[It\=o drift-diffusion process]\label{def:ito} A real or complex valued adapted stochastic process $X_t$ is called a \textit{It\=o drift-diffusion process} if there exists a Lebesgue integrable adapted $b_t$ and $W_t$ integrable adapted $\sigma_t$ such that $X_t$ satisfies the stochastic differential equation
$$dX_t = b_t \,dt + \sigma_t \,dW_t.$$
\end{definition}

For It\=o drift-diffusion processes there exists the following important analogue of the change of variable formula, which follows from robustness of Taylor expansions for stochastic differentials:

\begin{lemma}[It\={o}'s lemma]\label{lma:ito} Let $X_t$ be an It\=o drift-diffusion process and $f : \R \to \R$ twice differentiable. Then $f(X_t)$ is an It\=o drift-diffusion process such that $\P$ almost surely for any $T \geq 0$ we have
$$f(X_T) - f(X_0) = \int_0^T \Big(b_t f'(X_t) + \frac{\sigma_t^2}{2} f''(X_t) \Big) \, dt + \int_0^T \sigma_t f'(X_t) \, dW_t.$$
\end{lemma}

It\={o}'s lemma was given in this pathwise form in \cite[Theorem 3.3]{KS}. By using the definition of the complex valued stochastic integral, we can also obtain a complex valued It\=o's lemma as follows:

\begin{lemma}[Complex It\={o}'s lemma]\label{lma:itocomplex} Let $X_t$ be an It\=o drift-diffusion process and $f : \R \to \C$ twice differentiable. Then $f(X_t)$ is an It\=o drift-diffusion process such that for $\P$ almost surely for any $T \geq 0$ we have
$$f(X_T) - f(X_0) = \int_0^T \Big(b_t f'(X_t) + \frac{\sigma_t^2}{2} f''(X_t) \Big) \, dt + \int_0^T \sigma_t f'(X_t) \, dW_t.$$
\end{lemma}

\begin{proof}
We can write $f = f_1 + if_2$ for real valued twice differentiable $f_1,f_2 : \R \to \R$. Then the derivatives $f' = f_1' + if_2'$ and $f'' = f_1'' + if_2''$. Moreover, by It\=o's lemma (Lemma \ref{lma:ito}) we obtain for each $j = 1,2$ that
$$df_j(X_t) = \Big(b_t f_j'(X_t) + \frac{\sigma_t^2}{2} f_j'(X_t)\Big) \, dt + \sigma_t f_j'(X_t) \, dW_t.$$
Then by the convention $df(X_t) = df_1(X_t) + idf_2(X_t)$ this gives
$$df(X_t) = \Big(b_t f'(X_t) + \frac{\sigma_t^2}{2} f''(X_t)\Big) \, dt + \sigma_t f'(X_t) \, dW_t$$
as required.
\end{proof}

\subsection{Moment estimation} It\=o's lemma allows us to pass from integrals of the form $\int_0^T f(X_t) \, dt$  to $\int_0^T g(X_t) \, dW_t$ for functions $g$ obtained from derivatives of $f$. In our case we will end up trying to understand the higher moments of the stochastic integrals $\int_0^T g(X_t) \, dW_t$, which will tell us about the distribution of these integrals. A very standard tool to compute the moments in It\=o calculus are the \textit{It\=o isometry} and more general \textit{Burkholder-Davis-Gundy inequalities} (see \cite{BDG}), which allows us to pass from stochastic integrals to their quadratic variations (that just involve Lebesgue integral). 

\begin{lemma}[Burkholder-Davis-Gundy inequality]\label{lma:BDG} Let $X_t$ be a real valued $W_t$ integrable adapted process. Then for all $1 \leq p < \infty$ we have
$$\E\Big[\Big(\sup_{0 \leq s \leq 1}\Big|\int_0^s X_t \, dW_t\Big|\Big)^{2p}\Big] \leq 2 \sqrt{10p} \ \E \Big[\Big(\int_0^1 X_t^2 \, dt\Big)^p\Big].$$
\end{lemma}

This version with the constant $2\sqrt{10p}$ was given by Peskir \cite{peskir}.

\section{Proof of the main result}
\label{sec:proof}

\subsection{Preliminaries and overview of the proof}
\label{sec:mainsteps}

Let us now review how we will prove \eqref{eq:fourierdecay} and thus Theorem \ref{thm:main}. Fix $\xi = u(\cos \theta,\sin\theta)\in \R^2$ with modulus $u > 0$ and argument $\theta \in [0,2\pi)$. Notice that by the definition of the graph measure $\mu$, the Fourier transform has the form
$$\widehat{\mu}(\xi) = \int_0^1 \exp(iX_t) \, dt,$$
where $X_t$ is the real valued stochastic process
\begin{align}\label{eq:Xprocess} X_t := -2\pi u (t \cos \theta + W_t \sin \theta).\end{align}
The first observation is that $X_t$ is an adapted $W_t$ integrable process and in fact an It\=o drift-diffusion process (recall Definition \ref{def:ito}) satisfying
$$dX_t = b \,dt + \sigma \,dW_t$$
for deterministic and time independent coefficients $b = -2\pi u \cos \theta$ and $\sigma = -2\pi u \sin \theta$. The proof of bounding $\widehat{\mu}(\xi)$ will heavily depend on the value of the angle $\theta$ we have for $\xi$ and in particular how close the determining angle $\theta$ is to $0$, $\pi$ or $2\pi$ is with respect to $u^{-1/2}$. For this purpose, we define the notions of \textit{horizontal} and \textit{vertical} angles:

\begin{definition}[Horizontal and vertical angles]
Define the threshold angle
$$\theta_u := \min\{u^{-1/2},\tfrac{\pi}{4}\}.$$
Partition the angles $[0,2\pi)$ using $\theta_u$ into the \textit{horizontal angles}
$$H_u := [0,\theta_u] \cup [\pi - \theta_u,\pi+\theta_u] \cup [2\pi - \theta_u,2\pi).$$
and the \textit{vertical angles}
$$V_u := [0,2\pi) \setminus H_u.$$
In other words $H_u$ contains the $\theta_u$ neighborhoods of $0$ and $\pi$ on the circle mod $2\pi$ and $V_u$ the $\pi/2 - \theta_u$ neighborhoods of $\pi/2$ and $3\pi/2$ respectively, see Figure \ref{fig:angles}.
\end{definition}

\begin{figure}[ht]\label{fig:angles}
\includegraphics[scale=0.75]{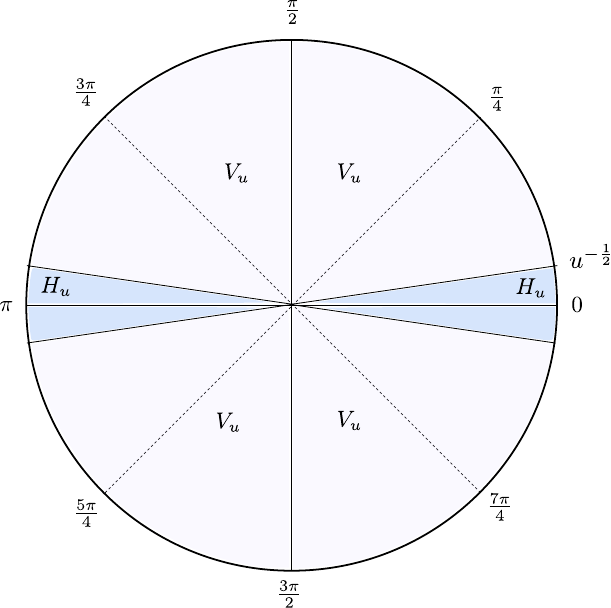}
\caption{Splitting of $[0,2\pi)$ to horizontal angles $H_u$ and the vertical angles $V_u$}
\end{figure}

The proof will split into two cases in Sections \ref{sec:horizontal} and \ref{sec:vertical} for bounding the Fourier transform $\widehat{\mu}(\xi)$ depending on whether $\theta \in H_u$ or $\theta \in V_u$.

\begin{itemize}
\item[(1)] Section \ref{sec:horizontal} concerns angles $\theta \in H_u$, that is, close to horizontal directions $0$ or $\pi$ and as mentioned in the introduction our main hope here is that the smallness (with respect to $u^{-1/2}$) of the diffusion component $b W_t$ will help us in transferring the decay of Lebesgue measure to the decay of $\widehat{\mu}$. This is where It\=o's lemma (see Lemma \ref{lma:itocomplex}) becomes crucial as it can be applied to the process $f(X_t)$ with the function $f(x) = \exp(ix)$.
\item[(2)] Section \ref{sec:vertical} handles the angles $\theta \in V_u$ and here the plan is to use the fact that we are $u^{-1/2}$ bounded away from horizontal angles to ignore the drift component $b t$ of the drift-diffusion process $X_t$ and apply Kahane's bound for these directions. This turns out to be possible due to a representation of the higher moments Kahane obtained in his result on Brownian images.
\end{itemize}

It turns out that in both  Sections \ref{sec:horizontal} and \ref{sec:vertical} we only obtain decay of the Fourier transform $\widehat{\mu}(\k)$ for $\k$ in an $\eps$-grid $\eps \Z^2$ for all small $\eps > 0$. Here the randomness will depend on $\eps > 0$ but thanks to an argument also used by Kahane in \cite{Ka1}, one can pass from this information to the full decay almost surely. See Section \ref{sec:completion} for the details. 

Let us now proceed to bound $|\widehat{\mu}(\xi)|$. In both Sections \ref{sec:horizontal} and \ref{sec:vertical} below we will end up bounding trigonometric functions with respect to $\theta_u$ and for this purpose we will need the following standard bounds, which we record here for convenience:

\begin{lemma}[Trigonometric bounds]\label{lma:trigonometric} We have the following bounds:
\begin{itemize}
\item[(1)] If $\theta \in H_u$, then
$$|\sin \theta| \leq u^{-1/2} \quad \text{and}\quad |\cos \theta| \geq \tfrac{1}{\sqrt{2}}.$$
\item[(2)] If $\theta \in V_u$, then
$$|\sin \theta| \geq \min\left\{\tfrac{2}{\pi}u^{-1/2}, \tfrac{1}{\sqrt{2}}\right\}.$$
\end{itemize}
\end{lemma}

\begin{proof}
For $\alpha \in [0,\tfrac{\pi}{2}]$ we have that both $\cos$ and $\sin$ are non-negative. Moreover, here $\frac{2}{\pi} \alpha \leq \sin \alpha \leq \alpha$. Thus for $\theta \in  [0,\theta_u]$ we have
$$\sin \theta \leq \theta \leq \theta_u \leq u^{-1/2} \quad \text{and}\quad \cos \theta \geq \cos \theta_u \geq \cos \tfrac{\pi}{4} = \tfrac{1}{\sqrt{2}}.$$
and for $\theta \in (\theta_u,\tfrac{\pi}{2}]$ as $\sin \tfrac{\pi}{4} = \frac{1}{\sqrt{2}}$ we obtain
$$\sin \theta \geq \min\left\{\tfrac{2}{\pi}u^{-1/2}, \tfrac{1}{\sqrt{2}}\right\}.$$
This gives the claim as we may reduce the estimates back to the estimates for $\theta \in [0,\tfrac{\pi}{2}]$ by using standard invariance identities for $\sin$ and $\cos$.
\end{proof}
 
\subsection{Horizontal angles}
\label{sec:horizontal}

When $\theta \in H_u$ we will first obtain the following estimate on $\eps$-grids:

\begin{lemma}\label{lma:ourbound} Fix $\eps > 0$. Almost surely there exists a random constant $C_\omega > 0$ such that for any $\k = u(\cos \theta,\sin\theta) \in \eps\Z^2 \setminus \{0\}$ with $\theta \in H_u$ we have
$$|\widehat{\mu}(\k)| \leq C_\omega |\k|^{-1/2}.$$
\end{lemma}

Given $\xi = u(\cos \theta,\sin\theta) \in \R^2 \setminus \{0\}$ and a realisation $(W_t)$ define a random time $T = T_\omega (\xi) \in [0,1]$, which the minimum value of $t \in [0,1]$ such that
$$X_t = \begin{cases}-2\pi \lceil u ( \cos \theta + W_1 \sin \theta) \rceil, & \text{if } X_1 \geq 0;\\
-2\pi \lfloor u ( \cos \theta + W_1 \sin \theta) \rfloor, & \text{if } X_1 < 0.\\
\end{cases}
$$
Such a time $T$ exists almost surely since $X_0 = 0$ and $X_t$ is almost surely continuous (since $W_t$ is almost surely continuous). Splitting the integral of $Z_t$ up into `complete rotations' and `what is left over', one obtains
$$\int_0^1 Z_t \, dt \  = \  \int_{0}^{T} Z_t \, dt  \ + \  \int_{T}^1 Z_t \, dt.$$
For the integral over $[T,1]$ we get the following estimate.

\begin{lemma} 
\label{lma:T1integral}
Almost surely there exists a random constant $C_\omega > 0$ such that for any $\xi = u(\cos\theta,\sin\theta)  \in \R^2\setminus\{0\}$ with $\theta \in H_u$ we have
$$ \Big|\int_{T}^1 Z_t \, dt\Big| \leq C_\omega |\xi|^{-1/2}.$$
\end{lemma}

\begin{proof}
Since $W_t$ is almost surely continuous, there almost surely exists a random constant $M_\omega > 1$ such that $W_t \in [-M_\omega,M_\omega]$ for all $t \in [0,1]$.  Define the real-valued process
$$Y_t := u (t \cos \theta + W_t \sin \theta)$$
so $X_t = -2\pi Y_t$. Suppose $X_1 \geq 0$. In this case $Y_T = \lceil Y_1 \rceil \leq 0$ and so $Y_1 + 1 \geq Y_T \geq Y_1$. Moreover, when $X_1 < 0$ we have $Y_T = \lfloor Y_1 \rfloor > 0$ and $Y_1 \geq Y_T \geq Y_1-1$. Thus no matter what is the sign of $X_1$ is, we always have almost surely
$$u (\cos \theta + W_1 \sin \theta) + 1 \geq u (T \cos \theta + W_T \sin \theta) \geq u (\cos \theta + W_1 \sin \theta) -1.$$
Therefore, in the case $\cos \theta > 0$ we obtain
$$T \geq 1 + W_1\frac{\sin\theta}{\cos \theta}  - W_T \frac{\sin \theta}{\cos \theta}- \frac{1}{u\cos \theta}$$
and when $\cos \theta < 0$ we have
$$T \geq 1 + W_1\frac{\sin\theta}{\cos \theta}  - W_T \frac{\sin \theta}{\cos \theta} + \frac{1}{u\cos \theta}.$$
Since $u \in H_u$ Lemma \ref{lma:trigonometric} together with $W_t  \in [-M_\omega,M_\omega]$ yields
$$T \geq 1 - 2\sqrt{2}M_\omega u^{-1/2} - \frac{\sqrt{2}}{u}.$$
Recalling $M_\omega > 1$ this gives
\[
\Big|\int_{T}^1 Z_t \, dt\Big|  \leq \int_{T}^1 |Z_t |\, dt = 1-T \leq 4 M_\omega u^{-1/2}
\]
as required.
\end{proof}

We now estimate the integral over $[0,T]$, which is where It\=o calculus comes into play.

\begin{lemma}
\label{lma:0Tintegral}
Fix $\eps > 0$. Almost surely there exists a random constant $C_\omega > 0$ such that for any $\k = u(\cos\theta,\sin\theta) \in \eps\Z^2 \setminus \{0\}$ with $\theta \in H_u$ we have
$$\Big|\int_{0}^{T} Z_t \, dt\Big| \leq C_\omega |\k|^{-1/2}.$$
\end{lemma}

To prove Lemma \ref{lma:0Tintegral}, we first need to compute the higher order moments of the random variable $\int_{0}^{T} Z_t \, dt$.

\begin{lemma}
\label{lma:ourmomentbound}
For any $p \in \N$ and $\xi = u(\cos\theta,\sin\theta) \in \R^2 \setminus \{0\}$ with $\theta \in H_u$ the $2p$th moment
$$\E \Big|\int_{0}^{T} Z_t \, dt\Big|^{2p} \leq   13 p^{1/2} \, 4^p |\xi|^{-p}.$$
\end{lemma}

\begin{proof}
Recall that
$$X_t = -2\pi u (t \cos \theta + W_t \sin \theta)$$ 
is an It\=o drift-diffusion process satisfying the stochastic differential equation
$$dX_t = b \,dt + \sigma \,dW_t$$
for deterministic and time independent coefficients $b = -2\pi u \cos \theta$ and $\sigma = -2\pi u \sin \theta$. Writing 
$$f(x) := \exp(ix), \quad x \in \R,$$ 
then $Z_t = f(X_t)$, $f'(x) = i\exp(ix)$ and $f''(x) = -\exp(ix)$. Thus by complex It\=o's lemma (see Lemma \ref{lma:itocomplex}) we have $\P$ almost surely
\begin{align}\label{inteq}f(X_T) - f(X_0) = (b i - \sigma^2/2) \int_0^T f(X_t) \, dt  + \sigma i \int_0^T f(X_t) \,dW_t.\end{align}
Note that $T_\omega \leq 1$ is random and only $\cF_1$ measurable, thus it is not a stopping time. However, as Lemma \ref{lma:itocomplex} is given \textit{pathwise}, that is, $\P$ almost surely It\=o's lemma holds for any time $T \geq 0$ then as $T_\omega$ is $\P$ almost surely well-defined, we have \eqref{inteq} almost surely. Since $X_0$ and $X_T$ are $2\pi$ multiples of integers by definition, we have $f(X_T) = f(X_0) = 1$.  Thus \eqref{inteq} gives
$$ \int_0^T f(X_t) \, dt = -\frac{\sigma i}{b i - \sigma^2/2} \int_0^T f(X_t) \,dW_t$$
Since $b$ and $\sigma$ are deterministic, this yields that the $2p$th moment
$$\E\Big|\int_0^T f(X_t) \,dt\Big|^{2p} = \Big|\frac{\sigma i}{b i - \sigma^2/2}\Big|^{2p}\E \Big| \int_0^T f(X_t) \,dW_t\Big|^{2p}.$$
Applying the  Burkholder–Davis–Gundy inequality (see Lemma \ref{lma:BDG}) for the process $\cos X_t$ gives
\begin{align*}\E\Big[\Big| \int_0^T \cos X_t \,dW_t\Big|^{2p}\Big] & \leq \E\Big[\Big( \sup_{0 \leq s \leq 1}\Big| \int_0^s \cos X_t \,dW_t\Big|\Big)^{2p}\Big]\\
& \leq 2 \sqrt{10 p} \ \E \Big( \int_0^1 \cos^2 X_t \,dt\Big)^{p} \\
&\leq 2 \sqrt{10 p}
\end{align*}
since $\cos^2 \leq 1$. Similar application for the process $\sin X_t$ gives
$$\E\Big[\Big| \int_0^T \sin X_t \,dW_t\Big|^{2p}\Big] \leq 2 \sqrt{10 p}.$$
By Euler's formula, we can write $f(X_t) = \cos X_t + i\sin X_t$ and so
$$\int_0^T f(X_t) \,dW_t = \int_0^T \cos X_t \,dW_t + i\int_0^T \sin X_t \,dW_t.$$
Hence
\begin{align*}\E\Big| \int_0^T f(X_t) \,dW_t\Big|^{2p} &=  \E\Big[\Big( \Big|\int_0^T \cos X_t \,dW_t\Big|^2 + \Big|\int_0^T \sin X_t \,dW_t\Big|^2\Big)^{p}\Big]\\
& \leq \E\Big[ 2^p\Big|\int_0^T \cos X_t \,dW_t\Big|^{2p}+ 2^p\Big|\int_0^T \sin X_t \,dW_t\Big|^{2p}\Big]\\
& = 2^p\Big( \E\Big[\Big| \int_0^T \cos X_t \,dW_t\Big|^{2p}\Big] + \E\Big[\Big| \int_0^T \sin X_t \,dW_t\Big|^{2p}\Big]\Big) \\
& \leq 2^p 4 \sqrt{10 p}.
\end{align*}
Moreover, as $\theta \in H_u$ we have by Lemma \ref{lma:trigonometric} that $\cos^2\theta \geq 1/2$ and $\sin^2 \theta \leq u^{-1}$. Hence
$$\Big|\frac{\sigma i}{b i - \sigma^2/2}\Big|^2 = \frac{\sigma^2}{b^2 + \sigma^4/4} \leq \frac{\sigma^2}{b^2} = \frac{4\pi^2 u^2 \sin^2 \theta}{4\pi^2 u^2 \cos^2 \theta} = \frac{\sin^2\theta}{\cos^2\theta} \leq 2 u^{-1}.$$
Therefore,
$$\E\Big|\int_0^T f(X_t) \,dt\Big|^{2p} \leq  4 \sqrt{10 p} \, 4^p \theta^{2p} \leq  13 p^{1/2} \, 4^p u^{-p}.$$
as required. 
\end{proof}

\begin{proof}[Proof of Lemma \ref{lma:0Tintegral}]
Fix $\eps >0$. Then for all $\k \in \eps\Z^2 \setminus \{0\}$ define the random variable
$$I(\k) := \Big(\int_{0}^{T} Z_t \, dt\Big) \cdot \chi_A(\k),$$
where $\chi_A$ is the indicator function on the set
$$A := \{\xi = u(\cos \theta,\sin\theta) \in \R^2 \setminus \{0\} : \theta \in H_u\}.$$ 
Note that $I(\k)$ is well-defined and finite since $| \int_{0}^{T} Z_t \, dt | \leq 1$ by $|\exp(ix)| = 1$ property. Lemma \ref{lma:ourmomentbound} now yields for any $\k \in \eps\Z^2 \setminus \{0\}$ and $p \in \N$ that
$$\E|I(\k)|^{2p} \leq 13 p^{1/2} \, 4^p |\k|^{-p}$$
as when $\k \notin A$ we have $I(\k) \equiv 0$. Write $p_\k = \lfloor \log |\k|\rfloor$. Then
$$\E \sum_{\k \in \eps\Z^2 \setminus \{0\}} |\k|^{-3} \frac{|I(\k)|^{2{p_\k}}}{13p_\k^{1/2}\, 4^{p_\k} |\k|^{-{p_\k}}} \leq \sum_{\k \in \eps\Z^2 \setminus \{0\}} |\k|^{-3} < \infty.$$
This means that the summands tend to $0$ almost surely as $|\k| \to \infty$ and so we can find a random constant $C_\omega > 0$ such that for all $\k \in \eps\Z^2 \setminus \{0\}$ we have
$$|\k|^{-3} \frac{|I(\k)|^{2{p_\k}}}{13p_\k^{1/2} \, 4^{p_\k} |\k|^{-{p_\k}}} \leq C_\omega.$$
Therefore, by possibly making $C_\omega$ bigger we obtain
$$|I(\k)| \leq C_\omega |\k|^{-1/2}.$$
This holds for each $\k \in \eps\Z^2 \setminus \{0\}$ so by definition of $I(\k)$ we have whenever $\k = u(\cos\theta,\sin\theta) \in \eps\Z^2 \setminus\{0\}$ with $\theta \in H_u$ that
$$\Big|\int_{0}^{T} Z_t \, dt\Big| \leq C_\omega u^{-1/2}$$
as claimed.
\end{proof}

We are now in position to complete the proof of Lemma \ref{lma:ourbound}.

\begin{proof}[Proof of Lemma \ref{lma:ourbound}]
Fix $\varepsilon>0$.  By the splitting
$$\widehat{\mu}(\xi) = \int_0^1 Z_t \, dt   =   \int_{0}^{T} Z_t \, dt   +   \int_{T}^1 Z_t \, dt$$
and Lemmas \ref{lma:T1integral} and \ref{lma:0Tintegral}, we have that almost surely there exists a constant $C_\omega > 0$ such that for all $\k = u(\cos \theta,\sin\theta) \in \eps\Z^2 \setminus \{0\}$ with $\theta \in H_u$ we have
$$|\widehat{\mu}(\k)| \leq \  \Big|\int_{0}^{T} Z_t \, dt\Big|  \ + \  \Big|\int_{T}^1 Z_t \, dt\Big| \leq C_\omega |\k|^{-1/2}$$
as required.
\end{proof}

\subsection{Vertical angles}
\label{sec:vertical}

In this section we apply Kahane's work to obtain Fourier decay estimates when $\theta \in V_u$.

\begin{lemma}\label{lma:kahanebound} Fix $\eps > 0$. Almost surely there exists a random constant $C_\omega > 0$ such that for any $\k = u(\cos \theta,\sin\theta) \in \eps\Z^2 \setminus \{0\}$ with $\theta \in V_u$ we have
$$|\widehat{\mu}(\k)| \leq C_\omega |\k|^{-1/2} \sqrt{\log|\k|}.$$
\end{lemma}

Let us discuss a few estimates Kahane obtained in \cite{Ka1}. Let $\nu$ be the push-forward of Lebesgue measure on $[0,1]$ under the map $t \mapsto W_t$, that is, $\nu$ is the Brownian image of Lebesgue measure. Kahane established the following:

\begin{theorem}[Kahane, page 255, \cite{Ka1}]\label{thm:kahane} Almost surely
$$|\widehat{\nu}(v)|  \leq O(|v|^{-1} \sqrt{\log|v|}) \quad \text{as } |v| \to \infty.$$
\end{theorem}

The key ingredient for the proof of Theorem \ref{thm:kahane} was based on establishing the following bound for the higher moments:

\begin{lemma}[Kahane, page 254, \cite{Ka1}, estimate (2)]\label{lma:kahaneimagemoment} There exists a constant $C > 0$ such that for any $v \in \R \setminus \{0\}$ and any $p \in \N$ we have
$$\E|\widehat{\nu}(v)|^{2p} \leq C^p p^p |v|^{-2p}.$$
\end{lemma}

We can use Lemma \ref{lma:kahaneimagemoment} to give a bound on the higher moments in our setting, but with the price that the exponent will increase from $-2p$ to $-p$.

\begin{lemma} \label{lma:kahaneprojectionsmoment}
There exists a constant $C > 0$ such that for any $p \in \N$ and $\xi = u(\cos\theta,\sin\theta) \in \R^2 \setminus \{0\}$ with $\theta \in V_u$ the $2p$th moment satisfies
$$\E|\widehat{\mu}(\xi)|^{2p} \leq C^p p^p |\xi|^{-p}.$$
\end{lemma}

\begin{proof}
Write $\t = (t_1,\dots,t_p)\in [0,1]^p$ and $d \t$ as the Lebesgue measure on $[0,1]^p$. Given $\t,\s \in [0,1]^p$, we denote
$$\phi(\t,\s) := \sum_{k = 1}^p (t_k - s_k), \quad \psi(\t,\s) := \sum_{k = 1}^p (W_{t_k} - W_{s_k}), \quad \text{and} \quad \Psi(\t,\s) := \E|\phi(\t,\s)|^2.$$
By the definition of $\mu_\theta$, $\mu$ and the Fourier-transform, and using the fact that the multivariate process 
$$X(\t,\s) := -2\pi \cos(\theta) \phi(\t,\s) -2\pi\sin(\theta)\psi(\t,\s)$$ 
is Gaussian with mean $-2\pi \cos(\theta)\phi(\t,\s)$ and variance $4\pi^2 \sin^2 (\theta) \Psi(\t,\s)$, we have through Fubini's theorem and the formula for the characteristic function that
\begin{align*}
\E|\widehat{\mu}(\xi)|^{2p} &= \E \int_{[0,1]^p}\int_{[0,1]^p} \exp(-2\pi i u (\cos(\theta) \phi(\t,\s)+\sin(\theta)\psi(\t,\s))) \, d\t\,d\s\\
&=  \int_{[0,1]^p}\int_{[0,1]^p} \E \exp( i u X(\t,\s)) \, d\t\,d\s\\
& = \int_{[0,1]^p} \int_{[0,1]^p} \exp(-2\pi i \cos(\theta)u \phi(\t,\s)-2\pi^2 |u \sin(\theta)|^2\Psi(\t,\s)) \, d\t\,d\s .
\end{align*}
Thus by taking absolute values inside the integrals, and observing that $|\exp(ix)| = 1$ for any $x \in \mathbb{R}$, we obtain
\begin{align}\E|\widehat{\mu}(\xi)|^{2p} \leq \int_{[0,1]^p} \int_{[0,1]^p} \exp(-2\pi^2 |u \sin(\theta)|^2\Psi(\t,\s)) \, d\t\,d\s. \label{eq:integral} 
\end{align}
On the other hand, by doing the expansion again for the Fourier transform $\widehat{\nu}$ of the image measure $\nu$ at $v := u \sin(\theta) \in \R \setminus \{0\}$ we see that
$$ \E|\widehat{\nu}(v)|^{2p} = \E\int_{[0,1]^p} \int_{[0,1]^p} \exp(-2\pi i v \psi(\t,\s)) \, d\t\,d\s = \int_{[0,1]^p} \int_{[0,1]^p} \exp(-2\pi^2 v^2\Psi(\t,\s)) \, d\t\,d\s,$$
which equals to \eqref{eq:integral}. Thus by Lemma \ref{lma:kahaneimagemoment} we have
\begin{align*}\E|\widehat{\mu}(\xi)|^{2p}  \leq C^p p^p |v|^{-2p}.\end{align*}
Since $\theta \in V_u$ we have $|\sin \theta| \geq \min\{\tfrac{2}{\pi}u^{-1/2}, \tfrac{1}{\sqrt{2}}\}$. When $|\sin \theta| \geq \tfrac{1}{\sqrt{2}}$ we obtain
$$C^p p^p |v|^{-2p} \leq (2C)^p p^p u^{-2p} \leq  (2C)^p p^p u^{-p}.$$
On the other hand, if $|\sin \theta| \geq \tfrac{2}{\pi}u^{-1/2}$ we have
$$C^p p^p |v|^{-2p} \leq C^p p^p (2u^{-1/2}/\pi)^{-2p} u^{-2p} \leq (C \pi^2/4)^p p^p u^{-p} $$
This completes the proof.
\end{proof}

Now we can complete the proof of Lemma \ref{lma:kahanebound} for vertical directions:

\begin{proof}[Proof of Lemma \ref{lma:kahanebound}]
Fix $\eps >0$. Then for all $\k = u(\cos\theta,\sin\theta)\in \eps\Z^2 \setminus \{0\}$ define the random variable
$$F(\k) := \widehat{\mu}(\k) \chi_B(\k),$$
where 
$$B := \{\xi = u(\cos\theta,\sin\theta) \in \R^2 \setminus \{0\}  : \theta \in V_u\}.$$ 
Now $F(\k)$ is a well-defined finite random variable as $|\widehat{\mu}(\k)| \leq 1$ for any $\k$. From Lemma \ref{lma:kahaneprojectionsmoment} we obtain for any $\k \in \eps\Z^2 \setminus \{0\}$ and $p \in \N$ that
$$\E|F(\k)|^{2p} \leq  C^p p^p |\k|^{-p}.$$
Write $p_\k = \lfloor \log |\k|\rfloor$. Then
$$\E \sum_{\k \in \eps\Z^2 \setminus \{0\}} |\k|^{-3} \frac{|F(\k)|^{2{p_\k}}}{C^{p_\k} {p_\k}^{p_\k} |\k|^{-p_\k}} \leq \sum_{\k \in \eps\Z^2 \setminus \{0\}} |\k|^{-3} < \infty.$$
This means that the summands tend to $0$ almost surely as $|\k| \to \infty$ and so we can find a random constant $C_\omega > 0$ such that for all $\k \in \eps\Z^2 \setminus \{0\}$ we have
$$|\k|^{-3} \frac{|F(\k)|^{2{p_\k}}}{C^{p_\k} {p_\k}^{p_\k} |\k|^{-p_\k}} \leq C_\omega.$$
Thus possibly making $C_\omega$ bigger this yields
$$|F(\k)| \leq C_\omega |\k|^{-1/2}\sqrt{\log |\k|}.$$
Now this holds for each $\k \in \varepsilon\Z^2 \setminus \{0\}$ so by definition of $F(\k)$ we have, whenever $\k = u(\cos\theta,\sin\theta) \in \eps \Z^2 \setminus \{0\}$ with $\theta \in V_u$, that
$$|\widehat{\mu}(\k)| \leq C_\omega |\k|^{-1/2}\sqrt{\log |\k|}$$
as claimed.
\end{proof}

\subsection{From lattices to $\R^2$}
\label{sec:completion}

We can now complete the proof of the main theorem. For this purpose, we need the following comparison lemma used by Kahane that allows to pass from convergence on lattices for Fourier transform to the whole space:

\begin{lemma}[Kahane, Lemma 1, page 252, \cite{Ka1}]\label{lma:comparison}
Suppose $\tau$ is a measure on $\R^2$ with support in $(-1,1)^2$. Suppose $\phi,\psi : (0,\infty) \to (0,\infty)$ that are decreasing as $t \to \infty$ with the doubling properties
$$\phi(t/2) = O(\phi(t)) \quad \text{and} \quad \psi(t/2) = O(\psi(t)) \quad \text{as } t\to \infty.$$
If the Fourier transform of $\tau$ along the integer lattice $\Z^2$ satisfies
$$|\widehat{\tau}(\n)| = O(\phi(|\n|) / \psi(|\n|)), \quad \text{as } |\n| \to \infty,$$ 
then
$$|\widehat{\tau}(\xi)| = O(\phi(|\xi|) / \psi(|\xi|)), \quad \text{as } |\xi| \to \infty.$$
\end{lemma}

\begin{proof}[Proof of Theorem \ref{thm:main}]
Combining Lemmas \ref{lma:kahanebound} and \ref{lma:ourbound} we have that for any $\eps > 0$, almost surely, there exists some random constant $C_\omega > 0$ such that for any $\k = u(\cos \theta,\sin\theta) \in \eps\Z^2 \setminus \{0\}$ we have
\begin{equation}\label{eq:boundformu}
| \widehat \mu (\k) | \leq C_\omega |\k|^{-1/2} \sqrt{ \log |\k| }.
\end{equation}
Define a measure $\tau_\eps$ on $\R^2$ such that 
$$\widehat{\tau_\eps}(\xi) := \widehat{\mu}(\eps \xi), \quad \xi \in \R^2.$$
By the almost sure continuity of $W_t$, we have that there exists a random constant $M_\omega > 0$ such that the diameter of the support of $\mu$ is at most $M_\omega$ almost surely. Taking an intersection of the events that \eqref{eq:boundformu} holds for $\eps = 1/n$ over all $n \in \N$ allows us to find a random $\eps = \eps_\omega > 0$ such that $\mu$ is supported on a set of diameter strictly less than $1/\eps$ and \eqref{eq:boundformu} holds almost surely with this $\eps$. This guarantees that the measure $\tau_\eps$ is supported on $(-1,1)^2$ and so applying Lemma \ref{lma:comparison} with the measure $\tau = \tau_\eps$ and the maps $\phi(t) := \sqrt{\log t}$ and $\psi(t) := t^{1/2}$ gives the claim.
\end{proof}

\section*{Acknowledgements}

We thank Tuomas Orponen for useful discussions during the preparation of this manuscript. We are also grateful to an anonymous referee for comments and suggestions which improved the focus of the paper.  Finally, we thank The Hebrew University of Jerusalem and  The University of Manchester for hosting us for research visits during the writing of this paper.

\end{document}